\documentclass[10pt, twosides]{amsart}

\usepackage{defs, amsaddr}

\title[Discrete time geometric PMP]{A Simple proof of the discrete time geometric Pontryagin maximum principle on smooth manifolds}
\thanks{The authors were supported in part by the grant 17ISROC001 from the Indian Space Research Organization.}

\author[M. Assif P K]{Mishal Assif P K}
\address{Department of Mechanical Engineering\\ IIT Bombay, Powai\\ Mumbai 400076, India.\\
\url{http://mishalassif.github.io}\\
}

\author[D. Chatterjee]{Debasish Chatterjee and Ravi Banavar}
\address{Systems \& Control Engineering\\ IIT Bombay, Powai\\ Mumbai 400076, India.\\
\url{http://www.sc.iitb.ac.in/~chatterjee}\\
\url{http://www.sc.iitb.ac.in/~banavar}\\
}
\email{\{mishal\_assif, dchatter, banavar\}@iitb.ac.in}

\keywords{optimal control, Pontryagin maximum principle, smooth manifolds}

\begin{document}

	\begin{abstract}
		We establish a geometric Pontryagin maximum principle for discrete time optimal control problems on finite dimensional smooth manifolds under the following three types of constraints: a) constraints on the states pointwise in time, b) constraints on the control actions pointwise in time, c) constraints on the frequency spectrum of the optimal control trajectories. Our proof follows, in spirit, the path to establish geometric versions of the Pontryagin maximum principle on smooth manifolds indicated in \cite{ref:Cha-11} in the context of continuous-time optimal control. 
    \end{abstract}

	\maketitle

	\section{Introduction}
		\label{sec:the intro}
		
The celebrated Pontryagin maximum principle (PMP) is a central tool in optimal control theory that provides first order necessary conditions for optimal controls. These necessary conditions can be used by algorithms to arrive at optimal control actions. The PMP was first introduced for continuous time control systems on $\R^n$ by Pontryagin and his students in \cite{ref:PonBolGamMis-62} and alternate proofs for the PMP later appeared in \cite{ref:Bol-71} and \cite{ref:LeeMar-67}. The discrete time Pontryagin maximum principle was developed primarily by Boltyanskii (see \cite{ref:Bol-75, ref:Bol-78} and the references therein) and discrete time is the setting of our current work.

While control systems evolving on $\R^n$ are the most common, systems with non-flat manifolds as configuration spaces also appear in a variety of engineering disciplines including robotics, quantum mechanical systems, and aerospace systems. Justifiably so, the continuous time PMP was extended to control systems evolving on smooth manifolds in a sequence of works from \cite{ref:Sus-98} through \cite{ref:AgrSac-04}; however the proofs given in these sources are quite complicated. The most recent proof of the geometric continuous time PMP appears in \cite{ref:Cha-11}; it deserves special mention because of its sheer simplicity. This work serves as a source of inspiration for our current article. Assuming the validity of the PMP on Euclidean spaces, in \cite{ref:Cha-11} the author derives the geometric version of the PMP by embedding the underlying manifold in a suitable Euclidean space and extending the optimal control problem on the manifold into an equivalent control problem on the Euclidean space, followed by appealing to the PMP on the Euclidean space, and finally translating the necessary conditions furnished by the PMP for the extended problem on the Euclidean space back to the manifold. This is the route that we follow here in the discrete time setting.

Almost all physical systems that are to be controlled naturally come with an array of constraints attached to them. In spite of this, there are few control techniques available that can actually compute constrained control actions in a tractable fashion. The continuous time PMP is no exception to this: numerical algorithms that seek to identify optimal controls from the necessary condition given by the PMP can handle control constraints rather efficiently. However, the necessary conditions given by the continuous time PMP for point-wise state constraints typically involve a measure, which is an infinite dimensional object, and numerical methods face grave difficulties in this setting. If one wants to include point-wise state constraints in the optimal control problem during the synthesis stage, it is better to perform some kind of discretization of the system first, and this is where the relevance of discrete time optimal control arises. A discrete time PMP on smooth manifolds can be employed by algorithms to solve state and control constrained control problems with relative ease.
 
In this article we address optimal control problems for discrete-time smooth control systems evolving on finite dimensional smooth manifolds in the presence of the following three important classes of constraints:
\begin{enumerate}[label=(\Roman*), leftmargin=*, widest=III, align=left]
	\item \label{constr:states} constraints on the states at each time instant,
	\item \label{constr:controlmag} constraints on the control magnitudes at each time instant, and
	\item \label{constr:controlfreq} constraints on the frequency of the control functions.
\end{enumerate}
We prove a discrete time PMP for control systems on smooth finite dimensional manifolds under the presence of the three classes of constraints of type mentioned above with the aid of three simple ingredients:
\begin{enumerate}[label=(Step \arabic*), leftmargin=*, widest=3, align=left]
	\item The Whitney embedding theorem, which is employed for embedding the smooth manifold in a suitable Euclidean space.
	\item A few basic extension theorems for smooth functions defined on embedded submanifolds, employed here to extend the original optimal control problem to the Euclidean space given by Step 1.
	\item The discrete time PMP on $\R^n$ under frequency constraints \cite{ref:ParCha-17}, employed to arrive at first order necessary conditions for optimality of the extended problem. 
\end{enumerate}

To our knowledge the only sources that discuss versions of the PMP for discrete time geometric optimal control problems are \cite{ref:PhoChaBan-18} and \cite{ref:KipGup-17}. The former establishes a PMP for a class of smooth control systems evolving on Lie groups under mild structural assumptions on the system dynamics. In contrast, in the present article we remove all such assumptions and present a neater version of the PMP with broader applicability using very different and simple tools. \cite{ref:KipGup-17} proves a PMP on smooth manifolds subject to the similar types of constraints that we consider here, but with the exception of the frequency constraints, and they do so under weaker assumptions on the smoothness of the cost, the constraints and the state transition maps. However the exposition in \cite{ref:KipGup-17} heavily relies on nontrivial tools of nonsmooth analysis, and is nowhere nearly as simple as the proof we present here. The frequency constraints treated in this article first appeared in \cite{ref:ParCha-17}, but the exposition there was limited to systems evolving on finite dimensional Euclidean spaces, as opposed to non-flat smooth manifolds.

	\section{Preliminaries}
		\label{sec:prelim}
We employ standard notation throughout the article: \(\Nz\) denotes the non-negative integers, \(\N\) the positive integers, \(\R\) the real numbers. If \(k\) is a positive integer, we let \([k] \Let \{1, \ldots, k\}\). The vector space \(\R^{\genDim}\) is always assumed to be equipped with the standard inner product \(\inprod{v}{v'} \Let v^\intercal v'\) for every \(v, v' \in \R^{\genDim}\), and we denote by $v^k$ the $k^{\text{th}}$ component of $v$. In the theorem statements, we use \(\dualSpace{\bigl(\R^{\genDim} \bigr)}\) to denote the dual space of \(\R^{\genDim}\) for the sake of precision; of course, \(\dualSpace{\bigl(\R^{\genDim}\bigr)}\) is isomorphic to \(\R^{\genDim}\) in view of the Riesz representation theorem. It is also assumed that $\R^{\genDim}$ is endowed with the standard partial order $\leq$; i.e., two vectors \(v, w \in \R^{\genDim}\) are related by \(v \leq w\) if and only if $v^i \leq w^i$ for all $i = 1, \ldots \genDim$.

	If $\M_1$ and $\M_2$ are smooth manifolds and $f:\M_1 \rightarrow \M_2$ is a smooth map, then $Tf : \TBundle{\M_1} \rightarrow \TBundle{\M_2}$ denotes the tangent lift of the map $f$ and $T^{*}f : \coTBundle{\M_2} \rightarrow \coTBundle{\M_1}$ denotes the cotangent lift of the map $f$. $Tf(x_0) : \TSpace{x_0}\M_1 \rightarrow \TSpace{f(x_0)}\M_2$ will denote the tangent lift of the map $f$ at $x_0$, and $T^{*}f(x_0) : \coTSpace{f(x_0)}\M_2 \rightarrow \coTSpace{x_0}\M_1$ will denote the cotangent lift of the map $f$ at $x_0$. Similarly, if $f: \M_1 \rightarrow \R$ is a smooth function, then $df : \M_1 \rightarrow \coTBundle{\M_1}$ will denote the differential of the function $f$. 
	
	In the rest of this section we shall define the basic concepts regarding convex sets and tents which appear later in the statement of the main result. For the sake of brevity, we will omit all proofs in this section.

	\begin{itemize}[label=\(\circ\), leftmargin=*]
		\item Let \(\genDim\) be a positive integer. Recall that a non-empty subset \(K \subset \R^{\genDim}\) is a \embf{cone} if for every \(y \in K\) and \(\alpha \ge 0\) we have \(\alpha y \in K\). In particular, \(0 \in \R^{\genDim}\) belongs to \(K\). A non-empty subset \(C \subset \R^{\genDim}\) is \embf{convex} if for every \(y, y' \in C\) and \(\theta \in [0, 1]\) we have \((1-\theta)y + \theta y' \in C\).
		\item A hyperplane \(\Gamma\) in \(\R^{\genDim}\) is an (\(\genDim -1\))-dimensional affine subset of \(\R^{\genDim}\). It can be viewed as the level set of a nontrivial linear function \(p : \R^{\genDim} \lra \R\). If \(p\) is given by \(p(\genVar) = \inprod{a}{\genVar}\) for some \(a (\neq 0) \in \R^{\genDim}\), then
			\[ \Gamma \Let \set[\big]{\genVar \in \R^{\genDim} \suchthat \inprod{a}{\genVar} = \alpha}.
			\]
		\item Let \(\OMG\) be a nonempty set in \(\R^{\genDim}\). By \(\affHull \OMG\) we denote the set of all affine combinations of points in \(\OMG\). That is,
			\[
				\affHull \OMG = \set[\bigg]{\sum_{i=1}^{k} \theta_{i} \genVar_{i} \suchthat \sum_{i=1}^{k} \theta_{i} = 1, \quad \genVar_{i} \in \OMG \quad \text{for } i = 1, \ldots, k, \text{ and } k \in \N}.
			\]
			In other words, \(\affHull \OMG\) is also the smallest affine set containing \(\OMG\). The relative interior \(\relInt \OMG\) of \(\OMG\) denotes the interior of \(\OMG\) relative to the affine space \(\affHull \OMG\).
		\item Let \(M\) be a convex set and \(\vertex \in M\). The union of all the rays emanating from \(\vertex\) and passing through points of \(M\) other than \(\vertex\) is a convex cone with vertex at \(\vertex\). The closure of this cone is called the \embf{supporting cone} of \(M\) at \(\vertex\).

	\end{itemize}

	We will now provide some definitions associated with the method of tents. Although we will not be directly using the method of tents in the proof of the main result, tents do appear in our final result, and so one needs to be familiar at least with the basic definition of what a tent is.
	\begin{definition}
		Let \(\OMG\) be a subset of \(\R^{\genDim}\) and let \(\vertex \in \OMG\). A convex cone \(Q \subset \R^{\genDim}\) with vertex \(\vertex\) is a \embf{tent} of \(\OMG\) at \(\vertex\) if there exists a smooth map \(\tentMap\) defined in a neighbourhood of \(\vertex\) such that:\footnote{The theory also works for \(\tentMap\) continuous.}
	\begin{enumerate}[leftmargin=*, align=left]
		\item \(\tentMap (\genVar) = \genVar + o  (\genVar - \vertex )\),\footnote{\label{fn:o-Notation} Recall the Landau notation \(\genFun (\genVar) = o (\genVar)\) that stands for a function \(\genFun(0) = 0\) and \(\lim_{\genVar \to 0} \frac{\abs{\genFun(\genVar)}}{\abs{\genVar}} = 0\).} and
		\item there exists \(\epsilon > 0\) such that \(\tentMap (\genVar) \in \OMG\) for \(\genVar \in Q \cap \ball{\vertex}\).
	\end{enumerate}
	\end{definition}

	We say that a convex cone \(\genTent \subset \R^{\genDim}\) with vertex at \(\vertex\) is a local tent of \(\OMG\) at \(\vertex\) if for every \(\genVar \in \relInt \genTent\) there is a convex cone \(Q \subset \genTent\) with vertex at \(\vertex\) such that \(Q\) is a tent of \(\OMG\) at \(\vertex\), \(\genVar \in \relInt Q\), and \(\affHull Q = \affHull \genTent\). Observe that if \(\genTent\) is a tent of \(\OMG\) at \(\vertex\), then \(\genTent\) is a local tent of \(\OMG\) at \(\vertex\).

	\par A tent to a set at a point is just a linear approximation of the set about the point. Intuitively, it is the set of directions along which it is possible to enter the set from the point. This intuition is reinforced through the following theorems which characterize the tents of some sets which appear commonly in applications. 

	\begin{theorem}[{{\cite[Theorem 8 on p.\ 11]{ref:Bol-75}}}]
	\label{th:tangent plane}
		Let \(\OMG\) be a smooth manifold in \(\R^{\genDim}\) and \(\genTent\) the tangent plane to \(\OMG\) at \(\vertex \in \OMG\). Then \(\genTent\) is a tent of \(\OMG\) at \(\vertex\).
	\end{theorem}

	\begin{theorem}[{{\cite[Theorem 9 on p.\ 12]{ref:Bol-75}}}]
	\label{th:half-space tent}
		Given a smooth function \(\genFun : \R^{\genDim} \lra \R\), let \(\vertex\) be such that \(\derivative{\genVar}{\genFun (\vertex )} \neq 0\). Define sets \(\OMG, \OMG_{0} \in \R^{\genDim}\) as
		\[
			\OMG \Let \set[\big]{ \genVar \in \R^{\genDim} \suchthat \genFun (\genVar ) \le \genFun (\vertex )}, \quad \OMG_{0} \Let \set[\big]{\vertex} \cup \set[\big]{\genVar \in \R^{\genDim} \suchthat \genFun (\genVar ) < \genFun (\vertex )}.
		\]
		Then the half-space \(\genTent\) given by the inequality \(\inprod{\derivative{\genVar}{\genFun (\vertex )}}{\genVar - \vertex} \le 0\) is a tent of both \(\OMG\) and \(\OMG_{0}\) at \(\vertex\).
	\end{theorem}

	\begin{theorem}[{{\cite[Theorem 10 on p.\ 12]{ref:Bol-75}}}]
	\label{th:supporting cone tent}
		Let \(\OMG \subset \R^{\genDim}\) be a convex set and let \(\genTent\) be its supporting cone at \(\vertex \in \OMG\). Then \(\genTent\) is a local tent of \(\OMG\) at \(\vertex\).
	\end{theorem}

	We will also need the following two theorems regarding embedded submanifolds for the proof of our main result.
	\begin{theorem}[{{\cite[Theorem 6.15 on p.\ 134]{ref:Lee-13}}}]
	\label{th: whitney_embed}
	Every smooth n-manifold admits an embedding into $\R^{2n+1}$ as a closed submanifold.
	\end{theorem}

	\begin{theorem}[{{\cite[Lemma 5.34 on p.\ 115]{ref:Lee-13}}}]
	\label{th: extension}
	Let $\M$ be an n-dimensional smooth manifold and $i: \M \rightarrow \R^N$ be a smooth embedding such that $i(\M)$ is a closed subset of $\R^N$. If $f:\M \rightarrow \R$ is a smooth function, there exists a smooth function $\tilde{f}: \R^N \rightarrow \R$ such that $f = \tilde{f}\circ i$.  
	\end{theorem}

	\section{Problem setup}
		\label{sec:opt prob}
Consider a discrete time control system evolving on an \(\sysDim\) dimensional smooth manifold \(\M\) described by
\begin{equation}
\label{e:gen sys}
	\state_{t+1} = \sysDyn (\state_{t}, \conInp_{t}) \quad \text{for } t = 0, \ldots, \horizon-1,
\end{equation}
where \(\state_{t} \in \M\), \(\conInp_{t} \in \R^{\conDim}\), and \( (\sysDyn)_{t=0}^{\horizon-1}\) is a family of maps such that \(\M \times \R^{\conDim} \ni (\xi, \mu) \mapsto \sysDyn[s](\xi, \mu) \in \M\) is continuously differentiable for each \(s = 0, \ldots, \horizon-1\). We emphasize that the condition $\state_t \in \M$ is not being enforced as a constraint; $\M$ is the natural state space of the control system \eqref{e:gen sys}. To wit, it is an intrinsic property of the family of the dynamics \((\sysDyn)_{t=0}^{\horizon-1}\) that any trajectory of \eqref{e:gen sys} starting on the manifold $\M$ lies entirely on $\M$.

Let \(\conInp \kth \Let (\conInp_{t} \kth)_{t = 0}^ {\horizon-1}\) denote the \(k^\text{th}\) control sequence, and \(\freqComp\) denote its discrete Fourier transform (DFT). The relationship between \(\freqComp\) and \(\conInp \kth\) is given by \cite[Chapter 7]{Stein-Shakarchi}:
\begin{equation}
\label{e:frequency components}
\begin{aligned}
	\freqComp \Let (\freqComp_{\xi})_{\xi=0}^{\horizon-1} = \biggl( \sum_{t = 0}^{\horizon-1} \conInp_{t}\kth \epower{-\ii 2\pi \xi t/\horizon} \biggr)_{\xi=0}^{\horizon-1} \quad & \text{for } \xi = 0, \ldots, \horizon-1 \\
	& \text{and } k = 1, \ldots, \conDim.
\end{aligned}
\end{equation}

Let $\horizon \in \N$ be fixed. The objective of this article is to provide first-order necessary conditions of a finite horizon constrained optimal control problem with continuously differentiable stage cost, terminal cost, and inequality and equality constraints. We write our abstract optimal control problem as:
\begin{equation}
\label{e:abstract prob}
\begin{aligned}
	& \minimize_{(\conInp_{t})_{t=0}^{\horizon -1}} && \sum_{t=0}^{\horizon-1} \cost (\state_{t}, \conInp_{t}) + \cost[T](\state_{T})\\
	& \sbjto &&
	\begin{cases}
		\text{dynamics \eqref{e:gen sys}},\\
		\text{state constraints at each stage } t = 0, \ldots, \horizon,\\
		\text{control constraints at each stage } t = 0, \ldots, \horizon-1,\\
		\text{constraints on frequency components of the control sequence.}
	\end{cases}
\end{aligned}
\end{equation}
where \(\M \ni \xi \mapsto \cost[T](\xi) \in \R\) and \(\M \times \R^{\conDim} \ni (\xi, \mu)\mapsto \cost(\xi, \mu) \in \R\) are continuously differentiable functions representing the terminal cost and stage cost at time \(t\) respectively, for \(t = 0, \ldots, \horizon-1\).

The three types of constraints considered in the optimal control problem \eqref{e:abstract prob} are as follows:
\begin{enumerate}[label=(\roman*), leftmargin=*, align=left, widest=iii]

	\item \emph{State constraints}: Let \( (\sysCons)_{t=1}^{\horizon}\) be a family of maps such that \(\M \ni \xi \mapsto \sysCons[s](\xi) \in \R^{\consDim[s]}\) is continuously differentiable for each \(s = 0, \ldots, \horizon\). We restrict the trajectories of the states \((\state_{t})_{t=0}^{\horizon}\) to be such that
	\begin{align*}
	\label{e:state constr.}
		\state_{0} &= \initState \text{ and } \sysCons[t](\state_{t}) \leq 0 \quad \text{for } t = 1, \ldots, \horizon.
	\end{align*}
	
	\item \emph{Control constraints}: \(\admControl_{t} \subset \R^{\conDim}\) is a given but otherwise arbitrary non-empty set for each \(t = 0, \ldots, \horizon\). We impose the requirement that the control action \(\conInp_{t}\) at stage \(t\) must lie in \(\admControl_{t}\):
	\begin{equation}
	\label{e:control constr.}
	\begin{aligned}
		\conInp_{t} \in \admControl_{t} \quad \text{for } t = 0, \ldots, \horizon-1.
	\end{aligned}
	\end{equation}
	
\item \emph{Frequency constraints}: For the $k^\text{th}$ component of the control sequence \(\conInp \kth\), we define \(\admFreq \kth \subset \C^{\horizon}\) to be the set of admissible frequency components of its discrete Fourier transform (DFT) \( \ \freqComp = (\freqComp_{\xi})_{\xi=0}^{\horizon-1}\). 
	For a vector \(v \in \C^{\horizon}\) we define its support as
	\[
		\support(v) \Let \set[\big]{ i \in \{1, \ldots, \horizon \} \suchthat v_{i} \not = 0},
	\]
	and stipulate that
	\begin{equation}
	\label{e:freq constr.}
		\freqComp \in \admFreq \kth \Let \set[\big]{ v \in \C^{\horizon} \suchthat \support(v) \subset W \kth},
	\end{equation}
	where \(W \kth \subset \{1, \ldots, \horizon \}\) represents the support for the admissible frequencies in the \(k^\text{th}\) control sequence. The sets \(\bigl(W \kth \bigr)_{k=1}^{\conDim}\) are assumed to be given as part of the problem specification. In effect the constraint \eqref{e:freq constr.} ensures that the frequency spectrum of the $k^\text{th}$ component of the control sequence does not contain any non-zero entries lying outside the set $W \kth$. Frequency constraints of the form \eqref{e:freq constr.} are required in applications where the designer is required to suppress certain undesirable frequency components in the control sequence. For instance, in satelites with flexible structures attached to them, damages to such structures may occur if their natural frequencies are excited in course of their motion. In such a situation it is essential to avoid the natural frequencies of the structures from the spectrum of the control trajectories, and such constraints are ensured precisely by constraints of the form \eqref{e:freq constr.}. It can be shown that \eqref{e:freq constr.} can be recast into a more condensed form as 
	\begin{equation}
	\label{e:freq assumptions}
		\freqConstr (\conInp_{0}, \ldots, \conInp_{\horizon-1}) = \sum_{t=0}^{\horizon-1} \freqDer \conInp_{t} = 0 \quad \text{,where \ } \freqDer[t] \in \R^{\freqDim \times \conDim}, \text{for some \ } \freqDim \in \N,
	\end{equation}
	where the matrices $\bigl( \freqDer[t] \bigr)_{t=0}^{\horizon-1}$ depend on the sets \(\bigl(W \kth \bigr)_{k=1}^{\conDim}\). For a more detailed discussion on how this transformation can be done, we refer the reader to \cite{ref:ParCha-17}. We shall refer to \(\freqConstr\) as our \embf{frequency constraint map}.
\end{enumerate}

The abstract optimal control problem \eqref{e:abstract prob} can now be formally written as:
\begin{equation}
\label{e:opt prob}
\begin{aligned}
& \minimize_{(\conInp_{t})_{t=0}^{\horizon -1}} && \sum_{t=0}^{\horizon-1} \cost (\state_{t}, \conInp_{t}) + \cost[\horizon](\state_{\horizon}) \\
	& \sbjto &&
	\begin{cases}
		\text{dynamics \eqref{e:gen sys}},\\
		\state_{0} = \initState, \\
		\sysCons[t](\state_{t}) \leq 0 \quad \text{for } t = 1, \ldots, \horizon,\\
		\conInp_{t} \in \admControl_{t} \quad \text{for } t = 0, \ldots, \horizon-1,\\
		\sum_{t=0}^{\horizon-1} \freqDer \conInp_{t} = 0
	\end{cases}
\end{aligned}
\end{equation}

An optimal solution \((\optCon)_{t=0}^{\horizon-1}\) of \eqref{e:opt prob} is a sequence in \(\prod_{i=0}^{\horizon-1} \admControl_{i}\), and it generates its corresponding optimal state trajectory \((\optState)_{t=0}^{\horizon}\) according to \eqref{e:gen sys}. The pair \( \bigl( (\optState)_{t=0}^\horizon, (\optCon)_{t=0}^{\horizon-1} \bigr)\) is called an \embf{optimal state-action trajectory}.

At this point we make note of the following notational convention in effect throughout the sequel: $\partialCoTLift{\state}{\sysDyn}{\genVar_{0}}{\conInp_{0}}{}$ will denote the cotangent lift of the map $\sysDyn(\cdot, \conInp_{0}): \M \rightarrow \M$ at $\genVar_{0}$ and $\partialCoTLift{\conInp_{}}{\sysDyn}{\genVar_{0}}{\conInp_{0}}{}$ will denote the cotangent lift of the map $\sysDyn(\genVar_{0}, \cdot): \R^{\conDim} \rightarrow \M$ at $\conInp_{0}$. Similarly, $\differential{\state}{\cost}{\genVar_{0},\conInp_{0}}$ will denote the differential of the map $\cost(\cdot, \conInp_{0}): \M \rightarrow \R$ at $\genVar_{0}$ and $\differential{\conInp_{}}{\cost}{\genVar_{0},\conInp_{0}}$ will denote the differential of the map $\cost(\genVar_{0}, \cdot): \R^{\conDim} \rightarrow \R$ at $\conInp_{0}$.

	\section{Main result}
		\label{sec:main result}
	The following theorem provides first order necessary conditions for optimal solutions of \eqref{e:opt prob}; it is the main result of this article. 
	
	\begin{theorem}[PMP on smooth manifolds]
	\label{th:main pmp}
		Let \(\bigl((\optState)_{t=0}^{\horizon}, (\optCon)_{t=0}^{\horizon-1} \bigr)\) be an optimal state-action trajectory for \eqref{e:opt prob}.
		Then there exist
		\begin{itemize}[label=\(\circ\), leftmargin=*]
			\item a trajectory \(\bigl(\adjDyn \bigr)_{t=1}^{\horizon} \subset \coTBundle{\M}\) with \( \adjDyn \in \coTSpace{\optState}\M \) for each $t$ (the adjoint trajectory),
			\item a sequence \(\bigl(\adjState \bigr)_{t=1}^{\horizon} \) with \(\adjState \in \dualSpace{\bigl( \R^{\consDim[t]} \bigr)} \) for each $t$ (the multipliers corresponding to the point-wise state constraints), and
			\item a pair \(\bigl(\adjCost, \adjFreq \bigr) \in \R \times \dualSpace{\big(\R^{\freqDim}\bigr)}\) (the abnormal multiplier and the multiplier corresponding to the frequency constraints, respectively),
		\end{itemize}
		satisfying the following conditions:
		\begin{enumerate}[label={\textup{(PMP-\roman*)}}, leftmargin=*, align=left, widest=iii]
			\item \label{pmp:non-negativity} non-negativity:
				\begin{quote}
					\(\adjCost \ge 0, \bigl(\adjState \bigr)_{t=1}^{\horizon}  \ge 0  ;\)
				\end{quote}
			\item \label{pmp:non-triviality} non-triviality:
				\begin{quote}
					the sequence \( \bigl(\adjState\bigr)_{t=1}^{\horizon}\) and the pair \( \bigl(\adjCost, \adjFreq \bigr) \) do not simultaneously vanish;
				\end{quote}
			\item \label{pmp:optimal dynamics} state and adjoint system dynamics
				\begin{alignat*}{2}
					& \optState[t+1] = \sysDyn(\optState, \optCon) \quad \text{ for } t = 0, \ldots, \horizon-1, \\				
					& \adjDyn[t] = \partialCoTLift{\state}{\sysDyn}{\optState}{\optCon}{\adjDyn[t+1]} - \adjCost \differential{\state}{\cost}{\optState, \optCon} - \coTLift{}{\sysCons}{\optState}{\adjState} \quad \text{ for } t = 1, \ldots, \horizon-1;
				\end{alignat*}
			\item \label{pmp:transversality} transversality:
				\begin{align*}
					\adjDyn[\horizon] = -\adjCost \differential{}{\cost[\horizon]}{\optState[\horizon]} - \coTLift{}{\sysCons[\horizon]}{\optState[\horizon]}{\adjState[\horizon]} ;
				\end{align*}
			\item \label{pmp:Hamiltonian VI} Hamiltonian maximization, point-wise in time,
				\begin{align*}
                    \inprod{\partialCoTLift{\conInp{}}{\sysDyn}{\optState}{\optCon}{\adjDyn[t+1]} - \adjCost \differential{\conInp{}}{\cost}{\optState,\optCon} + \freqDer^{T}\adjFreq }{\perturb{\conInp_t}} \le 0 
					\end{align*}
				whenever \( \optCon + \perturb{\conInp}_{t} \in \conTent ( \optCon ) \), where \(\conTent(\optCon)\) is a local tent at \(\optCon\) of the set \(\admControl_{t}\) of admissible actions;
				\item \label{pmp:complementary slackness} complementary slackness:
				\begin{align*}
				(\adjState)^{j}\sysCons^{j}(\optState) = 0 \quad \text{ for all } j \in \upto{\consDim}.
				\end{align*}
		\end{enumerate}
	\end{theorem}

	We present a complete but elementary proof of \autoref{th:main pmp} in  \S\ref{sec:proofs}. 

	\subsection*{Discussion}
	The rest of this section is devoted to a scrutiny of various facets of \autoref{th:main pmp} over a sequence of remarks.

	\begin{remark}
		One of the points of departure of \autoref{th:main pmp} from the Euclidean version of the PMP given in \cite[Theorem 3.1]{ref:ParCha-17} is \ref{pmp:Hamiltonian VI}. To wit, there appears to be no natural way of defining a Hamiltonian function analogous to the one given in \cite[Theorem 3.1]{ref:ParCha-17} in the geometric framework. It is also worth noting that the absence of a natural Hamiltonian is peculiar to the discrete time setting since a Hamiltonian function arises naturally in the continuous time geometric PMP. Indeed, in continuous time a key element in the definition of the Hamiltonian is the duality product between the adjoint trajectory lying on the cotangent bundle and the tangent vector field along the optimal state trajectory lying on the tangent bundle. In the discrete time geometric setting, however, the adjoint trajectory remains on the cotangent bundle of the manifold, but the tangent vector field is replaced by a discrete trajectory lying on the manifold itself. Since there is no natural product (pairing) between an element of the cotangent bundle and an element of the manifold, a natural definition of a Hamiltonian is difficult to arrive at.
	\end{remark}

	\begin{remark}
		It is not entirely appropriate to use the term ``Hamiltonian maximization condition'' for \ref{pmp:Hamiltonian VI}; we have not even defined a Hamiltonian function here, let alone derive a maximization condition. We still use this name for the condition because it is analogous to the actual Hamiltonian maximization condition in the continuous time counterpart of the PMP. However, such a maximization condition does hold under additional structural assumptions on the sets of admissible actions and regularity assumptions on the cost and transition maps. We refer the reader to \cite[\S 3.1]{ref:KipGup-17} for a detailed exposition on this.
	\end{remark}

	\begin{remark}
		The non-triviality condition \ref{pmp:non-triviality} stated here is somewhat non-standard. The non-triviality condition is usually stated as the adjoint trajectory \(\bigl(\adjDyn \bigr)_{t=1}^{\horizon} \) and the pair \( \bigl(\adjCost, \adjFreq \bigr) \) do not simultaneously vanish. The condition given in \ref{pmp:non-triviality} is slightly weaker than the standard non-triviality condition; if \( \bigl( \bigl(\adjDyn \bigr)_{t=1}^{\horizon}, \adjCost, \adjFreq \bigr) \) could not simultaneously vanish, then \( \bigl( \bigl(\adjState \bigr)_{t=1}^{\horizon}, \adjCost, \adjFreq \bigr) \) would not vanish simultaneously either, since if it did, by \ref{pmp:optimal dynamics} and \ref{pmp:transversality} \( \bigl( \bigl(\adjDyn \bigr)_{t=1}^{\horizon}, \adjCost, \adjFreq \bigr) \) would also vanish simultaneously.
		\par However, under the additional assumption of the constraints $\sysCons$ being regular (as defined in Definition \ref{def: regular}) at $\optState$, the condition stated in \ref{pmp:non-triviality} is equivalent to the standard non-triviality condition. Suppose\( \bigl( \bigl(\adjDyn \bigr)_{t=1}^{\horizon}, \adjCost, \adjFreq \bigr) \) did vanish simultaneously, then by \ref{pmp:optimal dynamics} and \ref{pmp:transversality}, we get that
		\begin{equation*}
			\coTLift{}{\sysCons}{\optState}{\adjState} = 0.
		\end{equation*}
		Also, by \ref{pmp:non-negativity} and \ref{pmp:complementary slackness}, we have
		\begin{align*}
			\adjState &\geq 0 \quad \text{and} \quad (\adjState)^{j}\sysCons^{j}(\optState) = 0 \; \text{ for all } j \in \upto{\consDim}.
		\end{align*}
		If the constraints $\sysCons$ are regular at $\optState$, the only $\adjState$ satisfying these three conditions will be $\adjState = 0$. Therefore, \( \bigl( \bigl(\adjState \bigr)_{t=1}^{\horizon}, \adjCost, \adjFreq \bigr) \) would also vanish simultaneously, contradicting \ref{pmp:non-triviality}.
	\end{remark}
	
	\begin{remark}
			First order necessary conditions for locally optimal solutions of finite dimensional constrained optimization problems (such as the KKT conditions) usually accompany a ``constraint qualification'' condition which at first glance is completely absent in our discussion. The difference between conditions \ref{pmp:non-negativity} - \ref{pmp:complementary slackness} and the standard KKT conditions is the presence of the abnormal multiplier $\adjCost$. Observe that \ref{pmp:non-negativity} only guarantees that $\adjCost \geq 0$, it is still possible that $\adjCost = 0$. When $\adjCost = 0$ we arrive at an ``abnormal'' situation where the necessary conditions \ref{pmp:non-negativity} - \ref{pmp:complementary slackness} no longer depend on either the stage costs or the terminal cost; this situation arises typically when the constraints are so tight that the cost functions play no key role in the determination of the optimizer(s). In the context of the PMP, constraint qualification conditions serve the purpose of strengthening the conditions of Theorem \ref{th:main pmp} by guaranteeing that $\adjCost$ is non-zero thereby precluding the aforementioned abnormal situation. Due to the presence of the abnormal multiplier, the PMP as presented in Theorem \ref{th:main pmp} holds regardless of any such constraint qualification conditions.
	\end{remark}

	\begin{remark}
		The conditions \ref{pmp:non-negativity} - \ref{pmp:complementary slackness} together constitute a well-defined two point boundary value problem, with \ref{pmp:transversality} along with the initial condition $\state_{0} = \initState$ giving the entire set of boundary conditions. Algorithms based on Newton step methods may be employed to solve this (algebraic) two point boundary value problem; see, eg., \cite[\S 2.4]{ref:Tre-12} for an illuminating discussion in the context of continuous-time problems. Fast solution techniques for two point boundary value problems is an active field of research.
	\end{remark}

	\section{Proof of the main result}
		\label{sec:proofs}
We present a proof of \autoref{th:main pmp} through the following steps:
\begin{itemize}
\item Step 1: The configuration manifold is embedded in a Euclidean space and we convert \eqref{e:opt prob} into an equivalent optimal control problem on this Euclidean space.
\item Step 2: First order necessary conditions for the equivalent problem on the Euclidean space are applied to the problem in Step 1.
\item Step 3: The necessary conditions in Step 2 are lifted back to the original manifold.
\end{itemize}

\subsection{Step 1}
\label{sec:main_proof_step1}

By \autoref{th: whitney_embed}, one can find a smooth embedding of \( \M \text{ \ in} \ \R^\embSysDim \), where \( \embSysDim = 2\sysDim + 1 \), such that the image of the embedding is a \emph{closed} subset of \( \R^{\embSysDim} \). Let \( \inclusionMap : \M \lra \R^{\embSysDim} \) denote such a smooth embedding.

We observe that \( i(\M) \times \admControl \ni (\xi, \mu) \mapsto \inclusionMap \circ \sysDyn (i^{-1}(\xi), \mu) \in \R^{\embSysDim} \) is a smooth map from a closed subset of \( \R^\embSysDim \times \R^\conDim \) to \( \R^\embSysDim \). Hence, it can be extended to a smooth map \( \embSysDyn : \R^{\embSysDim} \times \R^{\conDim} \lra \R^{\embSysDim} \) on the whole of \( \R^{\embSysDim} \). Similarly, \( \inclusionMap(\M) \ni \xi \mapsto \sysCons \circ i^{-1}(\xi) \in \R^{\consDim} \) and \( \inclusionMap(\M) \ni \xi \mapsto \cost \circ i^{-1}(\xi) \in \R \) are smooth maps from a closed subset of \( \R^{\embSysDim} \) to \( \R^{\consDim} \) and \( \R \), and hence they can be extended to corresponding smooth maps \( \embSysCons :  \R^{\embSysDim} \lra \R^{\consDim} \) and \( \embCost :  \R^{\embSysDim} \lra \R \).

Now let us define an extended optimal control problem
\begin{equation}
\label{e:emb opt prob}
\begin{aligned}
& \minimize_{(\conInp_{t})_{t=0}^{\horizon -1}} && \sum_{t=0}^{\horizon-1} \embCost (\embState_{t}, \conInp_{t}) + \embCost[T](\embState_{T})\\
	& \sbjto &&
	\begin{cases}
		\embState_{t+1} = \embSysDyn (\embState_{t}, \conInp_{t}) \quad \text{for } t = 0, \ldots, \horizon-1, \\
		\embState_{0} = \embInitState = \inclusionMap(\initState), \\
		\embSysCons[t](\embState_{t}) \leq 0 \quad \text{for } t = 1, \ldots, \horizon,\\
		\conInp_{t} \in \admControl_{t} \quad \text{for } t = 0, \ldots, \horizon-1,\\
		\sum_{t=0}^{\horizon-1} \freqDer \conInp_{t} = 0.
	\end{cases}
\end{aligned}
\end{equation}
If \(\bigl((\state_t)_{t=0}^{\horizon}, (\conInp_{t})_{t=0}^{\horizon-1}\bigr)\) is a feasible state-action trajectory of \eqref{e:opt prob}, then \(\bigl((i(\state_t))_{t=0}^{\horizon}, (\conInp_{t})_{t=0}^{\horizon-1}\bigr)\) is, clearly, a feasible state-action trajectory of \eqref{e:emb opt prob}. If \(\bigl((\embState_t)_{t=0}^{\horizon}, (\conInp_{t})_{t=0}^{\horizon-1}\bigr)\) is a feasible state-action trajectory of \eqref{e:opt prob}, then \((\embState_t)_{t=0}^{\horizon} \subset \inclusionMap(\M) \), since $\embState_{0} = \inclusionMap(\initState) \in \inclusionMap(\M)$ and $\embSysDyn$ is an extension of $\inclusionMap \circ \sysDyn$. So, the state-action trajectory \(\bigl((\state_t)_{t=0}^{\horizon}, (\conInp_{t})_{t=0}^{\horizon-1}\bigr)\) is a feasible solution of \eqref{e:opt prob} if and only if 
\(\bigl((i(\state_t))_{t=0}^{\horizon}, (\conInp_{t})_{t=0}^{\horizon-1}\bigr)\) is a feasible solution of \eqref{e:emb opt prob}. It is also straightforward to see that the cost incurred by the trajectory \(\bigl((\state_t)_{t=0}^{\horizon}, (\conInp_{t})_{t=0}^{\horizon-1}\bigr)\) is the same as that incurred by \(\bigl((i(\state_t))_{t=0}^{\horizon}, (\conInp_{t})_{t=0}^{\horizon-1}\bigr)\). Therefore, the state-action trajectory \(\bigl((\optState)_{t=0}^{\horizon}, (\optCon)_{t=0}^{\horizon-1}\bigr)\) is an optimal solution of \eqref{e:opt prob} if and only if \(\bigl((i(\optState))_{t=0}^{\horizon}, (\optCon)_{t=0}^{\horizon-1}\bigr)\) is an optimal solution of \eqref{e:emb opt prob}.

\subsection{Step 2}

\label{sec:main_proof_step2}

In this step we find first order necessary conditions satisfied by a solution of \ref{e:emb opt prob}.
To this end, we define the set \( \consSet \Let \set[\big]{\genVar \in \R^{\embSysDim} \suchthat \embSysCons(\genVar) \leq 0} \). For \( \genVar \in \consSet \) we define the \emph{active set} of indices \( \activeSet{t}{\genVar} \Let \set[\big]{i \in \upto{\consDim} \suchthat \embSysCons^{i}(\genVar) = 0} \).

\begin{definition}\label{def: regular}
Let \( \genFun : \M \rightarrow \R^{\genDim} \) be a smooth map from a smooth manifold $\M$ to $\R^{\genDim}$. We say that $\genFun$ is \textbf{regular} at $\genVar_{0} \in \M$ if the only $\mu \in \coTSpace{\genFun(\genVar_{0})}\R^{\genDim}$ satisfying the three conditions
	\begin{enumerate}[label=(\roman*), leftmargin=*, align=left, widest=iii]
		\item \( \coTLiftMap{}{\genFun}{\genVar_{0}}{\mu} = 0 \),
		\item \( \mu \geq 0 \), and
		\item \( \mu^i \genFun^i(\genVar_{0}) = 0 \text{ for all } i \in \upto{\genDim} \),
	\end{enumerate}
	is $\mu = 0$.
\end{definition}

\begin{proposition}
\label{p: constraint_tent}
	If $\embSysCons$ is regular at $\genVar_{0}$, then the closed convex cone
 	\[
		\consTent{t}{\genVar_{0}} \Let \set[\big]{\genVar \in \R^{\embSysDim} \suchthat \TLift{}{\embSysCons^{i}}{\genVar_{0}}{(\genVar - \genVar_{0})} \leq 0 \; \text{ for all }i \in \activeSet{t}{\genVar_{0}}}
	\]
	is a tent of $\consSet$ at $\genVar_{0}$. Moreover, the closed convex cone
	\[
		\consCone{t}{\genVar_{0}} \Let \set[\big]{\coTLift{}{\embSysCons}{\genVar_{0}}{\mu} \suchthat \mu \geq 0, \mu^i \embSysCons^i(\genVar_{0}) = 0 }
	\]
	is the dual cone of \( \consTent{t}{\genVar_{0}} \).
\end{proposition}
\begin{proof}
Define \( \consSet^{i} \Let \set[\big]{\genVar \in \R^{\embSysDim} \suchthat \embSysCons^{i}(\genVar) \leq 0}. \) Two cases arise. If $i \in \activeSet{t}{\genVar_0}$, then the closed convex cone 
\[
 \consTent{t}{\genVar_{0}}^{i} \Let \set[\big]{\genVar \in \R^{\embSysDim} \suchthat \TLift{}{\embSysCons^{i}}{\genVar_{0}}{(\genVar - \genVar_{0})} \leq 0} \text{ is a tent of $\consSet^{i}$ at $\genVar_0$. }
 \]
If $i \notin \activeSet{t}{\genVar_{0}}$, then $\genVar_{0}$ lies in the interior of $\consSet^{i}$, and therefore \( \consTent{t}{\genVar_{0}}^{i} = \R^{\embSysDim} \) is a tent of $\consSet^{i}$ at $\genVar_0$. The condition that $\embSysCons$ is regular at $\genVar_{0}$ is equivalent to \cite[Theorem 2]{ref:Bol-75} the cones $\consTent{t}{\genVar_{0}}^{i}$ being inseparable. Since $\consSet = \bigcap^{\consDim}_{i=1} \consSet^{i}$, it follows from \cite[Theorem 11]{ref:Bol-75} that \( \consTent{t}{\genVar_{0}} = \bigcap^{\consDim}_{i=1}\consTent{t}{\genVar_{0}}^{i} \) is a tent of $\consSet$ at $\genVar_{0}$, proving the first part of the claim.
\par The fact that $\consCone{t}{\genVar_{0}}$ is closed and that it is the dual cone of $\consTent{t}{\genVar_{0}}$ follows from Farkas' lemma as given in \cite[Proposition 2.3.1]{ref:Ber-09}.
\end{proof}

\par The notational conventions mentioned earlier will be used in this section also. $\partialCoTLift{\embState}{\embSysDyn}{\genVar_{0}}{\conInp_{0}}{}$ will denote the cotangent lift of the map $\embSysDyn(\cdot, \conInp_{0}): \R^{\embSysDim} \rightarrow \R^{\embSysDim}$ at $\genVar_{0}$ and $\partialCoTLift{\conInp_{}}{\embSysDyn}{\genVar_{0}}{\conInp_{0}}{}$ will denote the cotangent lift of the map $\embSysDyn(\genVar_{0}, \cdot): \R^{\conDim} \rightarrow \R^{\embSysDim}$ at $\conInp_{0}$. Similarly, $\differential{\embState}{\embCost}{\genVar_{0},\conInp_{0}}$ will denote the differential of the map $\embCost(\cdot, \conInp_{0}): \R^{\embSysDim} \rightarrow \R$ at $\genVar_{0}$ and $\differential{\conInp_{}}{\embCost}{\genVar_{0},\conInp_{0}}$ will denote the differential of the map $\embCost(\genVar_{0}, \cdot): \R^{\conDim} \rightarrow \R$ at $\conInp_{0}$. 

\begin{proposition}
\label{p: emb_pmp}
	Let \(\bigl((\embOptState)_{t=0}^{\horizon}, (\optCon)_{t=0}^{\horizon-1} \bigr)\) be an optimal state-action trajectory for \eqref{e:emb opt prob}.
		Then there exist
		\begin{itemize}[label=\(\circ\), leftmargin=*]
			\item a trajectory \(\bigl(\embAdjDyn \bigr)_{t=1}^{\horizon} \subset \coTBundle{\R^{\embSysDim}}\) such that \( \embAdjDyn \in \coTSpace{\embOptState}\R^{\embSysDim} \),
			\item a sequence \(\bigl(\embAdjState \bigr)_{t=1}^{\horizon} \) such that \(\embAdjState \in \dualSpace{\bigl( \R^{\consDim[t]} \bigr)} \), and
			\item a pair \(\bigl(\embAdjCost, \embAdjFreq \bigr) \in \R \times \dualSpace{\big(\R^{\freqDim}\bigr)}\),
		\end{itemize}
		satisfying the following conditions:
		\begin{enumerate}[label={\textup{(EPMP-\roman*)}}, leftmargin=*, align=left, widest=iii]
			\item \label{euc_pmp:non-negativity} non-negativity condition
				\begin{quote}
					\(\embAdjCost \ge 0, \bigl(\embAdjState \bigr)_{t=1}^{\horizon}  \ge 0  ;\)
				\end{quote}
			\item \label{euc_pmp:non-triviality} non-triviality condition
				\begin{quote}
					the multipliers \( \bigl(\embAdjState\bigr)_{t=1}^{\horizon}\) and the pair \( \bigl(\embAdjCost, \embAdjFreq \bigr) \) do not simultaneously vanish;
				\end{quote}
			\item \label{euc_pmp:optimal dynamics} state and adjoint system dynamics
				\begin{alignat*}{2}
					\embOptState[t+1] & = \embSysDyn(\embOptState, \optCon) \quad \text{ for } t = 0, \ldots, \horizon-1, \\
					\embAdjDyn[t] & = \partialCoTLift{\embState}{\embSysDyn}{\embOptState}{\optCon}{\embAdjDyn[t+1]} - \embAdjCost \differential{\embState}{\embCost}{\embOptState,\optCon} - \coTLift{}{\embSysCons}{\embOptState}{\embAdjState} \quad \text{ for } t = 1, \ldots, \horizon-1;
				\end{alignat*}
			\item \label{euc_pmp:transversality} transversality conditions
				\begin{align*}
					\embAdjDyn[\horizon] = -\embAdjCost \differential{}{\embCost[\horizon]}{\embOptState[\horizon]} - \coTLift{}{\embSysCons[\horizon]}{\embOptState[\horizon]}{\embAdjState[\horizon]} ;
				\end{align*}
			\item \label{euc_pmp:Hamiltonian VI} Hamiltonian maximization condition, point-wise in time,
				\begin{align*}
                    \inprod{\partialCoTLift{\conInp{}}{\embSysDyn}{\embOptState}{\optCon}{\embAdjDyn[t+1]} - \adjCost \differential{\conInp{}}{\embCost}{\embOptState,\optCon}+ \freqDer^{T}\embAdjFreq}{\perturb{\conInp_t}} \le 0 
				\end{align*}
				whenever \(\optCon + \perturb{\conInp}_{t} \in \conTent ( \optCon ) \), where \(\conTent(\optCon)\) is a local tent at \(\optCon\) of the set \(\admControl_{t}\) of admissible actions;
				\item \label{euc_pmp:complementary slackness} complementary slackness
				\begin{align*}
				(\embAdjState)^{j}\embSysCons^{j}(\embOptState) = 0 \quad \text{ for all } j \in 1,2,\ldots,\consDim.
				\end{align*}
		\end{enumerate}
\end{proposition}

\begin{proof}
	Suppose \( \embSysCons[s] \) is not regular at \( \embOptState[s] \) for some $s$. Then there exists $\mu \in \dualSpace{\bigl( \R^{\consDim[s]} \bigr)} $ such that \( \mu \neq 0, \coTLiftMap{}{\embSysCons[s]}{\embOptState[s]}{\mu} = 0, \mu \geq 0, \mu^i \embSysCons[s]^i(\embOptState[s]) = 0 \). We can now take $\embAdjState[s] = \mu, \embAdjState = 0$ for all $t \neq s, \embAdjCost = 0, \embAdjFreq = 0, \embAdjDyn = 0$ and the conditions of Proposition \ref{p: emb_pmp} hold trivially.
	\par If not, we can say that for all \(t \in \upto{\horizon}, \embSysCons[t] \) is regular at \( \embOptState[t] \). From Proposition \ref{p: constraint_tent} we know that the set 
	\[
		\consCone{t}{\genVar_{0}} \Let \set[\big]{\coTLift{}{\embSysCons}{\genVar_{0}}{\mu_{t}} \suchthat \mu_t^{i} \geq 0, \mu_t^{i}\embSysCons^{i}(\genVar_{0}) = 0 }
	\]
	is the dual cone of a tent to the set \( \consSet = \set[\big]{\genVar \in \R^{\embSysDim} \suchthat \embSysCons(\genVar) \leq 0} \). It follows now that Proposition \ref{p: emb_pmp} is just a restatement of \cite[Proposition C.6]{ref:ParCha-17} except for the condition \ref{euc_pmp:non-triviality}. Suppose the multipliers \( \bigl(\embAdjState\bigr)_{t=1}^{\horizon}\) and the pair \( \bigl(\embAdjCost, \embAdjFreq \bigr) \) vanish simultaneously, then from the transversality condition \ref{euc_pmp:transversality}, 
	\begin{align*}
	\embAdjDyn[\horizon] & = -\embAdjCost \differential{}{\embCost[\horizon]}{\embOptState[\horizon]} - \coTLift{}{\embSysCons[\horizon]}{\embOptState[\horizon]}{\embAdjState[\horizon]} = 0,
	\end{align*}
	and from the adjoint dynamics \ref{euc_pmp:optimal dynamics},
	\begin{align*}
	\embAdjDyn[t] & = \partialCoTLift{\embState}{\embSysDyn}{\embOptState}{\optCon}{\embAdjDyn[t+1]} - \embAdjCost \differential{\embState}{\embCost}{\embOptState,\optCon} - \coTLift{}{\embSysCons}{\embOptState}{\embAdjState} \\
	& = \partialCoTLift{\embState}{\embSysDyn}{\embOptState}{\optCon}{\embAdjDyn[t+1]} \quad \text{ for } t = 1, \ldots, \horizon-1.
	\end{align*}
	It follows that the trajectory \( \bigl(\embAdjDyn\bigr)_{t=1}^{\horizon}\) also vanishes. This contradicts the non-triviality condition given in \cite[Proposition C.6]{ref:ParCha-17}, and proves \ref{euc_pmp:non-triviality}.
\end{proof}

\subsection{Step 3}
\label{sec:main_proof_step3}

The necessary conditions we arrived at in Proposition \ref{p: emb_pmp} depends on both the particular embedding of $i : \M \mapsto \R^{\embSysDim}$ and the extensions \( \bigl( \embSysCons \bigr)_{t=1}^{\horizon}, \bigl( \embSysDyn \bigr)_{t=0}^{\horizon-1}, \bigl( \embCost \bigr)_{t=0}^{\horizon} \); this isn't desirable. In this step we finally arrive at the conditions in Theorem \ref{th:main pmp} from the conditions in Proposition \ref{p: emb_pmp}.

\begin{proof}[Proof of Theorem \ref{th:main pmp}]
Define \( \adjDyn \Let \coTLift{}{\inclusionMap}{\optState}{\embAdjDyn}, \adjState \Let \embAdjState, \adjCost = \embAdjCost \). Then 
\begin{align*}
\adjDyn[t] &= \coTLift{}{\inclusionMap}{\optState}{\embAdjDyn[t]} \\
&=\coTLiftMap{}{\inclusionMap}{\optState}\bigl(\partialCoTLift{\embState}{\embSysDyn}{\embOptState}{\optCon}{\embAdjDyn[t+1]} - \embAdjCost \differential{\embState}{\embCost}{\embOptState, \optCon}  - \coTLift{}{\embSysCons}{\embOptState}{\embAdjState}\bigr) \\
&= \partialCoTLift{\state}{(\embSysDyn \circ \inclusionMap)}{\optState}{\optCon}{\embAdjDyn[t+1]} - \embAdjCost \differential{\state}{(\embCost \circ \inclusionMap)}{\embOptState, \optCon} - \coTLift{}{(\embSysCons \circ \inclusionMap)}{\optState}{\adjState}.
\end{align*}
Since $\embCost$ and $\embSysCons$ are extensions of $\cost$ and $\sysCons$ respectively, \( \embCost \circ \inclusionMap = \cost, \embSysCons \circ \inclusionMap = \sysCons, \embSysDyn \circ \inclusionMap = \inclusionMap \circ \sysDyn \). Also, $\adjState = \embAdjState, \adjCost = \embAdjCost$ by definition. Therefore,
\begin{align*}
\adjDyn[t] &= \partialCoTLift{\state}{(\inclusionMap \circ \sysDyn)}{\optState}{\optCon}{\embAdjDyn[t+1]} - \adjCost \differential{\state}{\cost}{\embOptState, \optCon} - \coTLift{}{\sysCons}{\optState}{\adjState} \\
 &= \partialCoTLift{\state}{\sysDyn}{\optState}{\optCon}{\coTLift{\state}{\inclusionMap}{\optState[t+1]}{\embAdjDyn[t+1]}} - \adjCost \differential{\state}{\cost}{\embOptState, \optCon} - \coTLift{}{\sysCons}{\optState}{\adjState}.
\end{align*}
We now conclude that
\begin{align*}
\adjDyn[t] &= \partialCoTLift{\state}{\sysDyn}{\optState}{\optCon}{\adjDyn[t+1]} - \adjCost \differential{\state}{\cost}{\optState, \optCon}- \coTLift{}{\sysCons}{\optState}{\adjState}, \\
\adjDyn[\horizon] &= -\adjCost \differential{}{\cost[\horizon]}{\optState[\horizon], \optCon[\horizon]} - \coTLift{}{\sysCons[\horizon]}{\optState[\horizon]}{\adjState[\horizon]}.
\end{align*}
This proves \ref{pmp:optimal dynamics} and \ref{pmp:transversality}. Since \( \embCost(\embOptState, u) = \cost(\optState, u)\) for all \(u \in \R^{\conDim} \),
\begin{align*}
\differential{\conInp{}}{\embCost}{\embOptState, u} = \differential{\conInp{}}{\cost}{\optState, u}. 
\end{align*}
Since \( \inclusionMap \circ \sysDyn(\optState, u) = \embSysDyn(\embOptState, u)\) for all \(u \in \R^{\conDim} \),
\begin{align*}
\partialCoTLift{\conInp{}}{\embSysDyn}{\embOptState}{\optCon}{\embAdjDyn} &= \partialCoTLift{\conInp{}}{(\inclusionMap \circ \sysDyn)}{\optState}{\optCon}{\embAdjDyn} \\
&= \partialCoTLift{\conInp{}}{\sysDyn}{\optState}{\optCon}{\coTLift{\state}{\inclusionMap}{\optState}{\embAdjDyn}} \\
&= \partialCoTLift{\conInp{}}{\sysDyn}{\optState}{\optCon}{\adjDyn}.
\end{align*}
Therefore,
\begin{multline*}
\inprod{\partialCoTLift{\conInp{}}{\sysDyn}{\optState}{\optCon}{\adjDyn} - \adjCost \differential{\conInp{}}{\cost}{\optState,\optCon} + \freqDer^{T}\adjFreq }{\perturb{\conInp}} = \\
\inprod{\partialCoTLift{\conInp{}}{\embSysDyn}{\embOptState}{\optCon}{\embAdjDyn} - \embAdjCost \differential{\conInp{}}{\embCost}{\embOptState,\optCon} + \freqDer^{T}\embAdjFreq }{\perturb{\conInp}} \le 0
\end{multline*}
whenever \( \optCon + \perturb{\conInp}_{t} \in \conTent ( \optCon ) \), where \(\conTent(\optCon)\) is a local tent at \(\optCon\) of the set \(\admControl_{t}\) of admissible actions. This proves \ref{pmp:Hamiltonian VI}. \ref{pmp:non-negativity}, \ref{pmp:non-triviality}, and \ref{pmp:complementary slackness} are just restatements of \ref{euc_pmp:non-negativity}, \ref{euc_pmp:non-triviality}, and \ref{euc_pmp:complementary slackness}.
\end{proof}

	\providecommand{\bysame}{\leavevmode\hbox to3em{\hrulefill}\thinspace}
\providecommand{\MR}{\relax\ifhmode\unskip\space\fi MR }
\providecommand{\MRhref}[2]{%
  \href{http://www.ams.org/mathscinet-getitem?mr=#1}{#2}
}
\providecommand{\href}[2]{#2}

	\bigskip

\end{document}